\documentclass{amsart}

\usepackage{amsthm, amssymb}
\usepackage{cite}
\usepackage[dvipdfmx]{graphicx}

\theoremstyle{definition}
\newtheorem{theorem}{Theorem}[section]
\newtheorem{lemma}[theorem]{Lemma}
\newtheorem{proposition}[theorem]{Proposition}
\newtheorem{observation}[theorem]{Observation}
\newtheorem{fact}[theorem]{Fact}

\theoremstyle{definition}
\newtheorem{definition}[theorem]{Definition}

\newtheorem*{ac}{Acknowledgement}

\theoremstyle{remark}
\newtheorem{remark}[theorem]{Remark}

\newenvironment{rmenum}{
\begin{enumerate}

}
{\end{enumerate}}

\newcommand{\parcut}[2]{\delta_{#1}(#2)}

\newcommand{\reverse}[1]{#1^{-1}}
\newcommand{\nonneg}{\ge 0}
\newcommand{\nonnegz}{Z_{\nonneg}}
\newcommand{\inc}[1]{\chi_{#1}}

\newcommand{\signfs}[1]{\partial_{#1}}
\newcommand{\signfse}[2]{\partial_{#1}(#2)}

\newcommand{\ssim}[1]{\overset{#1}{\leftrightarrow}}

\newcommand{\gsim}[1]{\overset{#1}{\sim}}
\newcommand{\gsimb}[2]{\overset{#1}{\sim}_{#2}}
\newcommand{\sgpart}[2]{\mathcal{P}^{#2}(#1)}

\newcommand{\fcompb}[2]{\mathcal{G}(#1, #2)}
\newcommand{\sgpartb}[3]{\mathcal{P}^{#3}(#1, #2)}

\newcommand{\onevec}{1}

\title[Bidirected Graphs I]{Bidirected Graphs I: Signed General Kotzig-Lov\'asz Decomposition}
\author{Nanao Kita}
\address{National Institute of Informatics, 
2-1-2 Hitotsubashi, Chiyoda-ku, Tokyo, Japan 101-8430}
\email{kita@nii.ac.jp}
\date{\today}

\begin{document}

\begin{abstract} 
This paper is the first  from  a series of  papers 
 that establish a common analogue of 
 the strong component and  basilica decompositions for bidirected graphs. 
A bidirected graph is a graph in which a sign $+$ or $-$ is assigned to each end of each edge, 
and therefore is a common generalization of digraphs and signed graphs. 
Unlike digraphs, 
the reachabilities between vertices by directed trails and paths 
are not equal in general bidirected graphs. 
In this paper,  
we set up an analogue of the strong connectivity theory for bidirected graphs  
regarding directed trails, motivated by factor theory. 
We define the new concepts of  {\em circular connectivity}  and  {\em circular components} 
as generalizations of the strong connectivity and strong components. 
In our main theorem, 
we characterize the inner structure of each circular component;  
we define a certain binary relation between vertices in terms of the circular connectivity 
and prove that this relation is an equivalence relation. 
The nontrivial aspect of this structure arises from directed trails starting and ending with the same sign, 
and  is therefore characteristic to bidirected graphs that are not digraphs.  
This structure can be considered as an analogue of the general Kotzig-Lov\'asz decomposition, 
a known canonical decomposition in $1$-factor theory. 
From our main theorem, we also obtain a new result in $b$-factor theory, namely, 
 a $b$-factor analogue of the general Kotzig-Lov\'asz decomposition.

\end{abstract} 

\maketitle

\section{Introduction} 
\subsection{Background}
\subsubsection{General Kotzig-Lov\'asz decomposition for $1$-Factors} \label{sec:intro:back:1kl}
The {\em Kotzig-Lov\'asz decomposition}~\cite{kotzig1959a, kotzig1959b, kotzig1960, lovasz1972b, kitacathedral, kita2012partially, DBLP:conf/isaac/Kita12, lp1986, schrijver2003} is a {\em canonical decomposition} in $1$-matching theory~\cite{lp1986}. 
Canonical decompositions of graphs are fundamental tools in $1$-matching theory 
and  have the distinction of being  uniquely determined for each given graph.  
The classical Kotzig-Lov\'asz decomposition~\cite{kotzig1959a, kotzig1959b, kotzig1960, lovasz1972b, lp1986, schrijver2003}  
is a decomposition for {\em $1$-factor connected graphs}, 
a special class of graphs with $1$-factors, 
 and is known for producing many celebrated results, 
such as the two ear theorem and the tight cut lemma~\cite{lp1986, schrijver2003}. 
A general graph with $1$-factors can be considered as 
being built up by joining multiple $1$-factor connected graphs with edges 
under a certain rule; 
each of these component graphs is called a {\em $1$-factor component}. 
Comparably recently, 
the classical Kotzig-Lov\'asz decomposition was generalized 
for arbitrary graphs with $1$-factors~\cite{kitacathedral, kita2012partially, DBLP:conf/isaac/Kita12}; 
we call this the {\em general Kotzig-Lov\'asz decomposition} 
or sometimes just the {\em Kotzig-Lov\'asz decomposition}.
The general Kotzig-Lov\'asz decomposition 
is by itself a canonical decomposition 
and is also a piece of another more comprehensive canonical decomposition, the {\em basilica decomposition}~\cite{kitacathedral, kita2012partially, DBLP:conf/isaac/Kita12}. 
As such, 
the general Kotzig-Lov\'asz decomposition 
has produced new results in $1$-matching theory, 
such as a characterization of maximal barriers in general graphs~\cite{DBLP:conf/cocoa/Kita13, kita2012canonical} and new proofs of Lov\'asz's cathedral theorem for saturated graphs~\cite{kita2014alternative} and the tight cut lemma~\cite{kita2015graph}.

\subsubsection{Bidirected Graphs} 
Bidirected graphs~\cite{schrijver2003} are a common generalization of digraphs and signed graphs. 
A {\em bidirected graph} is a graph 
in which a sign $+$ or $-$ is assigned to each end of each edge. 
A {\em digraph} is a special bidirected graph 
 in which the ends of an edge have distinct signs. 
A {\em signed graph} is a graph in which a single sign is assigned to each edge, 
and  can be considered as a bidirected graph 
in which the ends of each edge have the same sign. 
The concept of bidirected graphs 
is first proposed by Edmonds and Johnson~\cite{Edmonds70matching:a, schrijver2003} 
to provide a general framework that integrates matchings and network flows. 
Various problems,  such as 
capacitated nonsimple $b$-matchings, 
in which $1$-matchings and simple $b$-matchings are included, 
  capacitated $b$-edge covers, 
and  minimum cost flows,  
are given a single unified formulation 
as an optimization problem over bidirected graphs. 
Bidirected graphs have also gained attention in the studies of nowhere-zero integral flows~\cite{bouchet1983nowhere, khelladi1987nowhere, xu2005flows} 
and totally unimodular matrices~\cite{delpia2009thesis, pitsoulis2009representability}. 
Needless to say, numerous studies exist regarding digraphs and signed graphs; 
see, e.g., Schrijver~\cite{schrijver2003} or Zaslavsky~\cite{zaslavsky2012mathematical}.

\subsection{Our Aim} 
We can naturally define a bidirected counterpart of directed paths and trails in digraphs:  
A {\em ditrail} in a bidirected graph 
is a trail such that, for each vertex term $v$, 
the signs of $v$ assigned with respect to $e_1$ and $e_2$ are distinct, 
where $e_1$ and $e_2$ are the edge terms immediately before and after $v$, respectively. 
A {\em dipath} in a bidirected graph can be defined as a ditrail in which no vertex is contained twice or more.

We should however note that 
general bidirected graphs have the following two features 
that digraphs do not possess, 
which make the structure of bidirected graphs rich and complicated. 
First, 
  general bidirected graphs 
 have four types of dipaths or ditrails, 
 whereas digraphs have only two types. 
That is, in digraphs, any dipath or ditrail 
clearly starts with $-$ and ends with $+$, or vice versa. 
In contrast, in bidirected graphs, 
there can be dipaths and ditrails 
that start and end with $-$ and $-$ or $+$ and $+$, 
in addition to  those with $+$ and $-$ or $-$ and $+$. 
Second, 
in bidirected graphs, 
 even if two vertices are connected by a ditrail, 
 it does not necessarily follow that these vertices are  connected by a dipath. 
In digraphs, 
if there is a directed trail from a vertex $u$ to a vertex $v$, 
then clearly there is a directed path from $u$ to $v$; 
however, this property fails for general bidirected graphs.

Thus, in this paper, 
we initiate a bidirected analogue of the strong connectivity theory with respect to ditrails. 
Our motivation for considering ditrails comes from $b$-factor theory~\cite{lp1986, schrijver2003}. 
Given a graph $G$ and a mapping $b$ from the vertex set onto the set of integers, 
a set of edges $F$ is a $b$-factor if the number of edges adjacent to each vertex $v$ is $b(v)$. 
When discussing $b$-factors,  
we are often required to detect ``alternating trails,''  
that is, a trail in a graph along which edges in $F$ and not in $F$ show up alternately, 
where $F$ is a given $b$-factor. 
This task can be considered as a task of detecting ditrails in the signed graph 
generated from $G$ by assigning $-$ and $+$ to each edge in $F$ and not in $F$, respectively. 
Because signed graphs are  special bidirected graphs, 
the theory for ditrails in bidirected graphs has  
implications for $b$-factor theory.

\subsection{Our Results in This Paper} 
\subsubsection{Main Theorem for Bidirected Graphs} 
We introduce the new concepts of  {\em circular connectivity} 
and {\em circular components} of bidirected graphs 
 as  generalizations of the strong connectivity and strong components of digraphs. 
Just as a digraph is made up of its strong components and the edges joining them, 
a bidirected graph is made up of its circular component and the edges joining them. 

In our main theorem,  we obtain the inner structure of each circular component in a bidirected graph. 
This structure is stated by a certain equivalence relation over the vertex set; 
we  first define a certain binary relation, $\ssim{\pm}$, between vertices 
 considering  whether two vertices are connected by a ditrail that starts and ends with the same sign, 
and  then prove that this relation is an equivalence relation, 
the quotient set of which, in fact, provides the inner structure of each circular component. 
The nontrivial aspect of this structure is characteristic to nondigraphic bidirected graphs.  
In  bidirected graphs, 
the vertex set of each circular component can consist of any number of equivalence classes. 
In contrast, for digraphs, this structure is trivial in that 
each equivalence class coincides with the vertex set of a strong component. 
In our subsequent papers, 
we show that these equivalence classes can be considered as the fundamental  units 
for considering the ditrail reachability between vertices. 

\subsubsection{Consequence for $b$-Factor Theory} 
This result for bidirected graphs contains an analogue of the general Kotzig-Lov\'asz decomposition for $b$-factors. 
The original general Kotzig-Lov\'asz decomposition for $1$-factors 
is the quotient set of a certain equivalence relation $\gsim{}$ defined over the vertex set. 
We define a $b$-matching analogue of this equivalence relation, $\gsimb{\pm}{b}$,  
and prove that $\gsimb{\pm}{b}$ is also an equivalence relation in the following way: 
Given a graph $G$ and a $b$-factor $M$, 
create a bidirected graph $G^M$ from $G$ by 
assigning the sign $-$ to every end of every edge in $M$ 
and the sign $+$ to every end of every edge not in $M$.  
It is easily observed that this relation $\gsimb{\pm}{b}$ of $G$ coincides 
with the relation $\ssim{\pm}$ of $G^M$, 
and thus our main result for bidirected graphs immediately proves the claim. 
The counterpart of $1$-factor components for $b$-matchings 
is the concept known as {\em $b$-flexible components}~\cite{kita2016dulmage}. 
The $b$-flexible components of $G$ correspond to the circular components of the auxiliary bidirected graph $G^M$, 
and accordingly,   
this $b$-matching analogue of the general Kotzig-Lov\'asz decomposition in fact 
provides the inner structure of each $b$-flexible component.

Under this result for $b$-factors, 
our result for bidirected graphs can be considered 
as a generalization of the general Kotzig-Lov\'asz decomposition. 
Hence, we name this structure described by $\ssim{\pm}$ 
the {\em general Kotzig-Lov\'asz decomposition} for bidirected graphs.  

Considering the relationship between the original Kotzig-Lov\'asz decomposition and $1$-factor theory, 
we can expect from this new result for $b$-factors 
further new consequences in $b$-factor theory. 
For example, 
the $b$-factor analogues of the two ear theorem and tight cut lemma 
and their further consequences might be derived.

\subsection{Further Consequences} 
This paper is in fact the first  from a series of papers 
that establish the circular connectivity theory of bidirected graphs~\cite{kitascathii, kitascathiii}. 
Just as the strong component decomposition for digraphs 
tells us  how an entire graph is structured from its strong components, 
how an entire bidirected graph is made up of its circular components 
will be revealed in our subsequent papers. 
As is also the case for digraphs, 
this entire structure can be stated 
in terms of a partial order between circular components. 
However, 
the structure will be again much  richer and more complicated 
for bidirected graphs. 

Here, our general Kotzig-Lov\'asz decomposition for bidirected graphs, 
the inner structure of each circular component, 
 will turn out to be related to the entire graph structure.

This whole theory of circular connectivity provides a bidirected analogue of the basilica decomposition 
mentioned in Section~\ref{sec:intro:back:1kl}. 
The original basilica decomposition for $1$-factors is a canonical decomposition applicable 
for any graph with $1$-factors, 
and consists of three main concepts regarding $1$-factor components: 
the general Kotzig-Lov\'asz decomposition, 
which provides the inner structure of each $1$-factor component; 
the basilica order, which is a partial order between $1$-factor components; 
and the relationship between the two. 
In our circular connectivity theory, 
analogues of these concepts are provided regarding circular components. 
This theory can be considered as a common generalization of 
the strong component decomposition and basilica decomposition. 
As is also the case in this paper, 
this theory derives a new result for $b$-factors, that is, 
the $b$-factor analogue of the basilica decomposition.

The general Kotzig-Lov\'asz decomposition for bidirected graphs can be computed in polynomial time. 
The algorithm will be introduced in another paper by us~\cite{kitascathiii}.

\subsection{Organization of Paper} 
The remainder of this paper is organized as follows. 
In Section~\ref{sec:definition}, 
we explain basic notation and definitions used in this paper. 
In Section~\ref{sec:vs}, 
we make some remarks about the disparities between digraphs and general bidirected graphs regarding dipaths and ditrails. 
In Section~\ref{sec:circomp}, 
we introduce the concept of circular connectivity of bidirected graphs,  
and show that this is a generalization of the strong connectivity of digraphs.  
In Section~\ref{sec:skl}, 
we prove our main result, 
the analogue of  general Kotzig-Lov\'asz decomposition for bidirected graphs. 
In Section~\ref{sec:bkl}, 
we show that the result in Section~\ref{sec:skl} 
easily derives a $b$-factor analogue of a known result  
in $1$-factor theory, i.e., the general Kotzig-Lov\'asz decomposition.

\section{Notation} \label{sec:definition} 
\subsection{Graphs} 
For basic notation, we mostly follow Schrijver~\cite{schrijver2003}. 
In the following, we list exceptions or  nonstandard definitions that are used. 
Let $G$ be an (undirected) graph. 
The vertex set and edge set of $G$ are denoted by $V(G)$ and $E(G)$, respectively. 
Let $X\subseteq V(G)$. 
The set of edges joining $X$ and $V(G)\setminus X$ 
is denoted by $\parcut{G}{X}$. 
The subgraph of $G$ induced by $X$ is denoted by $G[X]$. 
The graph $G[V(G)\setminus X]$ is denoted by $G-X$. 
As usual, we often denote a singleton $\{x\}$ by $x$.

Let $u,v\in V(G)$. 
A {\em walk} from $u$ to $v$ is a sequence 
$(s_1, \ldots, s_k)$, 
where $k$ is an odd number with $k\ge 1$, 
such that
\begin{rmenum} 
\item 
for each odd $i\in \{1, \ldots, k\}$, 
$s_i$ is a vertex of $G$, for which $s_1 = u$ and $s_k = v$, and 
\item 
for each even $i\in \{1, \ldots, k\}$, 
$s_i$ is an edge of  $G$ that connects $s_{i-1}$ and $s_{i+1}$. 
\end{rmenum} 
Vertices and edges from a walk may not be distinct. 
A {\em trail} is a walk whose edges are all distinct. 
Let $W$ be a walk $(s_1, \ldots, s_k)$. 
We say that $W$ is {\em closed } if $s_1 = s_k$. 
For  vertices $s_i$ and $s_j$ from $W$ with $i \le j$, 
$s_iWs_j$ denote the walk $(s_i, \ldots, s_j)$. 
The {\em reverse walk} of $W$ 
is the sequence $(s_k, \ldots, s_1)$ 
and is denoted by $\reverse{W}$. 
Let $U$ be another walk $(s_k,\ldots, s_l)$, where $k\le l$. 
Then, $W + U$ 
denotes the walk $(s_1,\ldots, s_k, \ldots, s_l)$, 
namely, the concatenation of $W$ and $U$.

\subsection{Bidirected Graphs} 
A {\em bidirected graph} is a graph in which 
each end of each edge is assigned a sign $+$ or $-$. 
The precise definition is as follows: 
A {\em bidirected graph} is a graph $G$ endowed 
with two mappings $\signfs{+}$ and $\signfs{-}$ over $E(G)$ 
such that, 
for each $e\in E(G)$ with ends $u$ and $v$,  
\begin{rmenum} 
\item $\signfse{+}{e}$ and $\signfse{-}{e}$ 
are subsets of $\{ u, v\}$ one of which can be empty, 
\item $\signfse{+}{e} \cup \signfse{-}{e} = \{u, v\}$, and 
\item if $u\neq v$ then $\signfse{+}{e} \cap \signfse{-}{e} = \emptyset$. 
\end{rmenum}
We say that $\alpha\in \{+, -\}$  is a {\em sign of $u\in V(G)$ over $e \in E(G)$} 
if $u \in \signfse{\alpha}{e}$ holds. 
Bidirected graphs are a generalization of (usual) digraphs. 
A {\em digraph} is a bidirected graph 
in which $|\signfse{+}{e} |  = |\signfse{-}{e} | = 1$ for every edge $e$; 
we call a bidirected graph with this property {\em digraphic}.  
We use this adjective when we wish to note that we are considering  a digraph as a special bidirected graph.

Let $G$ be a bidirected graph. 
Let $W$ be a trail $(s_1, \ldots s_k)$, where $k \ge 1$. 
We call $W$ a {\em ditrail} 
if, for each odd $i \in \{1, \ldots, k\} \setminus \{1, k\}$, 
the signs of the vertex $s_i$ over the edges $s_{i-1}$ and $s_{i+1}$ are distinct. 
We call $W$ an $(\alpha, \beta)$-ditrail, where $\alpha, \beta \in \{+, -\}$, 
if $k > 1$ holds, and $W$ is a ditrail in which the sign of $s_1$ over $s_2$ is $\alpha$, 
whereas the sign of $s_k$ over $s_{k-1}$ is $\beta$. 
If $k = 1$, we also define $W$ as an $(\alpha, \beta)$-ditrail for any $\alpha, \beta \in \{+ ,-\}$ 
with $\alpha \neq \beta$. 
Note that if $W$ is an $(\alpha, \beta)$-ditrail, 
then $W^{-1}$ is a $(\beta, \alpha)$-ditrail. 
We call $W$ a {\em cyclic ditrail} 
if $W$ is a closed $(\alpha, -\alpha)$-ditrail for some $\alpha\in\{+, -\}$. 
We call a ditrail a {\em dipath} 
if no vertex is contained more than once. 

\section{Digraphs Versus Bidirected Graphs} \label{sec:vs}

In this section, we make some remarks 
 comparing digraphs and nondigraphic bidirected graphs regarding ditrails or dipaths. 
Note  the following  observations for digraphic bidirected graphs. 

\begin{observation} \label{obs:trail2path}
Let $G$ be a digraphic bidirected graph, and let $\alpha, \beta\in\{+, -\}$. 
Then, for any $s,t\in V(G)$, 
there is an $(\alpha, \beta)$-ditrail from $s$ to $t$ 
if and only if 
there is an $(\alpha, \beta)$-dipath from $s$ to $t$. 
\end{observation} 

That is, in considering the strong connectivity of digraphs, 
we are not required to distinguish between ditrails and dipaths. 
However, 
for general bidirected graphs, 
the connectivities by ditrails and  by dipaths are distinct; 
clearly, 
two vertices connected by $(\alpha, \beta)$-ditrails 
do not necessarily imply that they are connected by  $(\alpha, \beta)$-dipaths. 
In this paper, we study the connectivity of bidirected graphs by ditrails. 
As we will see in Section~\ref{sec:bkl},  
this study has implications for factor theory. 

\begin{observation} \label{obs:balance} 
In a digraphic bidirected graph,  any ditrail is a $(-, +)$- or $(+, -)$-ditrail. 
\end{observation}

That is, in the strong connectivity theory of digraphs, 
we need only consider $(-, +)$-ditrails between two vertices. 
In contrast, in general digraphs, there are also $(-, -)$- and $(+, +)$-ditrails. 
This variety that is peculiar to nondigraphic bidirected graphs 
 leads to  a new structure that we introduce in the following sections.

\section{Circularly Connected Components of Bidirected Graph} \label{sec:circomp}

We now introduce the new concept of {\em circular connectivity} in bidirected graphs 
and prove that this is a generalization of the strong connectivity in digraphs.  
In this section, unless stated otherwise, let $G$ be a bidirected graph 
with respect to ditrails. 

\begin{definition} 
We say that an edge $e\in E(G)$ is {\em circular} 
if there is a cyclic ditrail that contains $e$. 
We say that  vertices $u$ and $v$ are {\em circularly connected} 
if there is a path between $u$ and $v$ whose edges are all circular. 
A bidirected graph is {\em circularly connected} if every two vertices are circularly connected. 
A {\em circularly connected component} or {\em circular component} of $G$ 
is a maximal circularly connected subgraph of $G$.  
\end{definition} 

An alternative way to define a circular component of $G$ 
is as follows: 
Let $F \subseteq E(G)$ be the set of circular edges of $G$. 
A circular component is a subgraph of the form $G[V(C)]$, 
where $C$ is a connected component of the subgraph of $G$ determined by $F$. 
A bidirected graph consists of its circular components,
 which are disjoint,  
and edges joining distinct circular components, which are not circular.

The circular connectivity of bidirected graphs 
is a generalization of the strong connectivity of  digraphs. 
In a digraph, two vertices $u$ and $v$ are {\em strongly connected} 
if there are $(-, +)$-dipaths from $u$ to $v$ and from $v$ to $u$.  
A maximal strongly connected subgraph is called a {\em strongly connected component} 
or {\em strong component}.  
The next statement is a basic fact.  
\begin{fact}[see, e.g., Schrijver~\cite{schrijver2003}] \label{fact:strong2disjoint}
No two distinct strong components share vertices. 
Accordingly, a digraph is made up of its strong components, which are mutually disjoint, 
and the edges joining them. 
\end{fact}

Under Observation~\ref{obs:trail2path}, 
two vertices $u$ and $v$ are strongly connected if and only if 
there are $(-, +)$-ditrails from $u$ to $v$ and from $v$ to $u$. 
Accordingly, 
\begin{quote}
an edge $uv$ is circular 
if and only if $u$ and $v$ are strongly connected. 
\end{quote}
This further implies the following statement. 

\begin{observation} 
Let $G$ be a digraphic bidirected graph. 
Let $u,v\in V(G)$. 
Then, $u$ and $v$ are strongly connected if and only if $u$ and $v$ are circularly connected. 
\end{observation}

That is, 
the circular connectivity of bidirected graphs is a generalization of 
the strong connectivity of digraphs, 
and circular components are a generalization of strong components.

\section{General Kotzig-Lov\'asz Decomposition for Bidirected Graphs} \label{sec:skl}

In this section, we prove the main result of this paper: 
the analogue of the general Kotzig-Lov\'asz decomposition for bidirected graphs. 
In this section, unless stated otherwise, 
let $G$ be a bidirected graph and let $\alpha \in \{ +, -\}$. 
In the following, we define a binary relation $\ssim{\alpha}$ and then prove in Theorem~\ref{thm:skl} 
that this relation is an equivalence relation,  
the quotient set of which in fact provides the inner structure of each circular component. 

\begin{definition} 
Define a binary relation $\ssim{\alpha}$ over $V(G)$ as follows: 
For $u, v\in V(G)$, 
we let $u\ssim{\alpha} v$ if $u$ and $v$ are identical or if 
$u$ and $v$ are circularly connected 
and there is no $(\alpha, \alpha)$-ditrail between $u$ and $v$. 
\end{definition} 

We prove in the following that $\ssim{\alpha}$ is an equivalence relation.

\begin{proposition} \label{prop:reach}
Let $G$ be a circularly connected bidirected graph, 
let $\alpha\in \{+, -\}$, and let $s\in V(G)$. 
Then, for any $t\in V(G)$, 
there exists $\beta\in \{+, -\}$ 
such that $G$ has an $(\alpha, \beta)$-ditrail  from $s$ to $t$. 
\end{proposition} 
\begin{proof} 
Define $S\subseteq V(G)$ as follows: 
Let $x\in S$ if there exists $\beta \in \{+, -\}$ 
such that $G$ has an $(\alpha, \beta)$-ditrail from $s$ to $x$. 
We prove $S = V(G)$ in the following. 
Suppose, to the contrary, $S \subsetneq V(G)$. 
Because $G$ is circularly connected, 
there is a circular edge $uv\in \parcut{G}{S}$, where $u\in S$ and $v\in V(G)\setminus S$. 
Because $u\in S$ holds, 
there is an $(\alpha, \beta)$-ditrail $P$ from $s$ to $u$, where $\beta$ is either $+$ or $-$. 
Because $v\not\in S$ holds, the vertex $v$ is not contained in $P$, 
and therefore $P + (u, uv, v)$ is a trail from $s$ to $v$. 
This further implies that the sign of $u$ over the edge $uv$ is $\beta$. 
Let $\gamma$ be the sign of $v$ over $uv$.  
Because $uv$ is circular, 
there is a $(-\beta, -\gamma)$-ditrail $Q$ from $u$ and $v$. 
If $P$ and $Q$ do not share an edge, 
then $P + Q$ is an $(\alpha, -\gamma)$-ditrail from $s$ to $v$, 
 which implies $v\in S$; this is a contradiction.  
Hence, consider now the case in which $P$ and $Q$ share edges. 
Trace $P$ from $s$, and let $wz$ be the first encountered edge shared by $E(Q)$. 
Without loss of generality, let $w$, $wz$, and $z$ appear in this order over $P$. 
First, assume that $w$, $wz$, and $z$ also appear in this order over $Q$. 
Then, $sPz + zQv$ is an $(\alpha, -\gamma)$-ditrail from $s$ to $v$, 
that is, $v\in S$ holds, which is a contradiction. 
Assume now that $z$, $zw$, and $w$ appear in this order over $Q$. 
Then, $sPz + z\reverse{Q}u$ is an $(\alpha, -\beta)$-ditrail from $s$ to $u$. 
Therefore, $sPz + z\reverse{Q}u + (u, uv, v)$ is an $(\alpha, \gamma)$-ditrail from $s$ to $v$, 
which is again a contradiction. 
This completes the proof. 
\end{proof}

\begin{remark} 
The value of $\beta$ in Proposition~\ref{prop:reach} is not exclusive. 
That is, for two vertices $s$ and $t$, 
there may be both $(\alpha, +)$- and $(\alpha, -)$-ditrails from $s$ to $t$. 
\end{remark}

\begin{theorem} \label{thm:skl} 
Let $G$ be a bidirected graph, 
and let $\alpha\in \{+, -\}$. 
Then, $\ssim{\alpha}$ is an equivalence relation over $V(G)$. 
\end{theorem} 
\begin{proof} 
Reflexivity and symmetry are obvious from the definition. 
Let $u,v, w\in V(G)$ be vertices 
with $u\ssim{\alpha} v$ and $v\ssim{\alpha} w$. 
We prove $u\ssim{\alpha} w$ in the following. 
If any two from $u, v, w$ are identical, the statement obviously holds. 
Hence, assume that these are mutually distinct. 
It is obvious that $u, v, w$ are mutually circularly connected. 
Suppose $u\ssim{\alpha} w$ does not hold, 
and let $P$ be an  $(\alpha, \alpha)$-ditrail from $u$ to $w$. 
From Proposition~\ref{prop:reach},  
there is an $(\alpha, -\alpha)$-ditrail $Q$ from $v$ to $u$. 
Trace $Q$ from $v$, and let $x$ be the first encountered vertex in $P$. 
Then, either $vQx + xPw$ or $vQx + x\reverse{P}u$ is an $(\alpha, \alpha)$-ditrail from $v$ to 
$w$ or  $u$, respectively. 
This contradicts either $u\ssim{\alpha} v$ or $v\ssim{\alpha} w$. 
Therefore, $u\ssim{\alpha} w$ is proved. 
\end{proof}

For each $\alpha\in \{+, -\}$, 
we denote as $\sgpart{G}{\alpha}$ the family of equivalence classes of $\ssim{\alpha}$, 
and call this family the {\em general Kotzig-Lov\'asz decomposition} 
or simply the {\em Kotzig-Lov\'asz decomposition} of  the bidirected graph $G$ regarding the sign $\alpha$.  

Let $H$ be a circular component of $G$. 
From the definition of the equivalence relation, 
the family $\{ S\in \sgpart{G}{\alpha}: S \subseteq V(H)\}$ 
forms a partition of $V(H)$. 
Therefore, 
this decomposition of a bidirected graph 
can be considered as providing the inner structure of each circular component. 
If $G$ is digraphic, 
then $\{ S\in \sgpart{G}{\alpha}: S \subseteq V(H)\}$ 
coincides with $\{ V(H)\}$. 
The nontrivial aspects of the general Kotzig-Lov\'asz decomposition  
are characteristic to  bidirected graphs that are not digraphs.

Note that this inner structure of circular components  
is determined in the context of the entire bidirected graph; 
recall that the definition of $\ssim{\alpha}$ is given considering the entire $G$. 
It is easily confirmed that the family $\{ S\in \sgpart{G}{\alpha}: S \subseteq V(H)\}$  
 is not equal to $\sgpart{H}{\alpha}$ in general 
but is a refinement of $\sgpart{H}{\alpha}$. 

\section{Consequences for $b$-Factor Theory} \label{sec:bkl}

\subsection{Definitions regarding $b$-Factors} \label{sec:bkl:def}
Let $G$ be a graph, and let $b: V(G)\rightarrow \nonnegz$. 
A set of edges $M \subseteq E(G)$ is a {\em $b$-matching} 
if $|\parcut{G}{v} \cap M | \le b(v)$ holds for each $v\in V(G)$. 
A $b$-matching is {\em maximum } if it has the maximum number of edges.  
A $b$-matching $M$ is {\em perfect} if $|\parcut{G}{v} \cap M | = b(v)$ holds for each $v\in V(G)$. 
A perfect $b$-matching is also known as a {\em $b$-factor}.  
A $b$-factor is a maximum $b$-matching, however the converse does not necessarily hold. 
We say that $G$ is {\em $b$-factorizable} 
if it has a  $b$-factor.  
We denote $b$ by $\onevec$ if $b(v)$ is $1$ for every $v\in V(G)$. 
Thus, $\onevec$-matchings and $\onevec$-factors are the most fundamental concepts in matching theory, 
 also known as {\em matchings} or {\em perfect matchings}.  

Now, let $G$ be $b$-factorizable. 
An edge $e \in E(G)$ is {\em $b$-allowed} 
if $G$ has a $b$-factor that contains $e$; 
otherwise, we say that $e$ is {\em $b$-forbidden}.

A $b$-allowed edge is {\em $b$-flexible} 
if $G$ also has a $b$-factor that does not contain $e$; 
otherwise, we say that $e$ is {\em $b$-essential}.

For a subgraph $H$, 
$b|_H$ denotes the mapping $V(H)\rightarrow \nonnegz$ 
such that $b|_H(v) := b(v) - k$ for each $v\in V(H)$, 
where $k$ denotes the number of $b$-essential edges 
that connect $v$ and $V(G)\setminus V(H)$.  
For two mappings $b_1, b_2$: $V(G)\rightarrow \nonnegz$, 
$b_1 + b_2$ and $b_1 - b_2$ denote the mappings $V(G)\rightarrow Z$ 
such that $(b_1 + b_2)(v) = b_1(v) + b_2(v)$ 
and $(b_1 - b_2)(v) = b_1(v) - b_2(v)$ for each $v\in V(G)$. 
As usual, 
we utilize the associativity of these operations over mappings. 
Given a vertex $u\in V(G)$, 
$\inc{u}$ denotes the mapping  $V(G)\rightarrow \nonnegz$  
such that $\inc{u} = 1$ and $\inc{v} = 0$ for each $v\in V(G)\setminus \{u\}$.

Vertices $u,v\in V(G)$ are {\em $b$-flexibly connected} (resp. {\em $b$-factor connected}) 
if there is a path between $u$ and $v$ in which every edge is $b$-flexible 
(resp. {\em $b$-allowed}). 
A subgraph $H$ of $G$ is {\em $b$-flexibly connected} 
(resp. {\em $b$-factor connected}) 
if any two vertices in $H$ are $b$-flexibly connected (resp. $b$-factor connected) 
in $G$. 
A maximal $b$-flexibly connected subgraph  
is called a {\em $b$-flexibly connected component} 
or {\em $b$-flexible component} 
(resp. a {\em $b$-factor connected component} or {\em $b$-factor component}) of $G$. 
The graph $G$ consists of its $b$-flexible components, 
which are disjoint, and edges joining distinct $b$-flexible components, 
which are $b$-forbidden or $b$-essential. 
A $b$-factor component consists of 
some $b$-flexible components and edges joining them.

A set $M \subseteq E(G)$ is a $b$-factor 
if and only if 
it is the union of the set of $b$-essential edges and 
a set $M^*$ of the form $M^*= \bigcup \{ M_C: C\in\fcompb{G}{b} \}$, 
where $M_C$ is a $b|_C$-factor of  $C\in \fcompb{G}{b}$. 
The $b$-flexible components form the most fine-grained set of subgraphs that satisfies this property. 
Hence, 
$b$-flexible components can be considered as 
the fundamental units for $b$-factors. 

A similar property can be stated for $b$-factor components; 
that is, a set $M \subseteq E(G)$ is a $b$-factor if and only if 
it is a  union of a $b|_C$-factor, where $C$ is taken over every $b$-factor component. 
In $1$-factor theory, 
 $1$-factor components have been used as fundamental units of a graph. 
 For example, the Dulmage-Mendelsohn decomposition theory, a classical canonical decomposition for bipartite graphs, 
 relates to a poset over $1$-factor components. 
 More examples can be found in the basilica decomposition theory and 
 the polyhedral studies of $1$-factors.

\subsection{General Kotzig-Lov\'asz Decomposition for $b$-Factors} \label{sec:bkl:bkl} 

We now show that Theorem~\ref{thm:skl} implies the analouge of general Kotzig-Lov\'asz decomposition 
for $b$-factors. 
In this section, 
unless stated otherwise, 
let $G$ be a $b$-factorizable graph, where $b: V(G)\rightarrow \nonnegz$.  
We define a binary relation $\gsimb{\pm}{b}$ 
that is uniquely determined for $G$ and $b$.  
This relation is proved to be an equivalence relation 
by generating an auxiliary bidirected graph $G^M$ from $G$ 
and some $b$-factor $M$ and then applying Theorem~\ref{thm:skl} for $G^M$.

\begin{definition} 
We define binary relations $\gsimb{-}{b}$ and $\gsimb{+}{b}$ over $V(G)$ as follows:
For $u, v\in V(G)$, let $u\gsimb{-}{b} v$ (resp. $u\gsimb{+}{b} v$) 
if $u$ and $v$ are identical or if 
$u$ and $v$ are $b$-flexibly connected and  $G$ does not have a $b-\inc{u}-\inc{v}$-factor 
(resp. $b + \inc{u} + \inc{v}$-factor). 
\end{definition} 

In the following,  
we show that $\gsimb{-}{b}$ and $\gsimb{+}{b}$  are equivalence relations.

\begin{definition} 
Let $M\subseteq E(G)$. 
We denote by $G^M$ the bidirected graph obtained by 
endowing mappings $\signfs{+}$ and $\signfs{-}$ as follows: 
 for each $e\in E(G)$ with ends $u, v\in V(G)$, 
 let $\signfse{+}{e} = \emptyset$ and $\signfse{-}{e} = \{u,v\}$ if $e$ is an edge from $M$; 
 otherwise, let $\signfse{+}{e} = \{u, v\}$ and $\signfse{-}{e} = \emptyset$.  
\end{definition} 

The next lemma is easily confirmed 
from classical observations regarding $b$-matchings. 

\begin{lemma} \label{lem:reduction}  
Let $G$ be a $b$-factorizable graph, where $b: V(G)\rightarrow \nonnegz$, 
and let $M$ be a $b$-factor of $G$. 
Then, for each $u,v\in V(G)$, 
$u\gsimb{\alpha}{b} v$ holds in $G$ 
if and only if $u\ssim{\alpha} v$ holds in $G^M$. 
\end{lemma} 

Under Lemma~\ref{lem:reduction}, Theorem~\ref{thm:skl} thus derives the next theorem. 

\begin{theorem} \label{thm:bkl}
Let $G$ be a $b$-factorizable graph, where $b: V(G)\rightarrow \nonnegz$. 
Then, $\gsimb{-}{b}$ and $\gsimb{+}{b}$ are equivalence relations over $V(G)$. 
\end{theorem} 

For each $\alpha \in \{+, -\}$, 
we denote as $\sgpartb{G}{b}{\alpha}$ 
the family of equivalence classes of $\gsimb{\alpha}{b}$. 
We call $\sgpartb{G}{b}{+}$ and $\sgpartb{G}{b}{-}$ the {\em general Kotzig-Lov\'asz decompositions} 
or simply {\em the Kotzig-Lov\'asz decompositions} 
 of the $b$-factorizable graph $G$ {\em  by addition} and {\em subtraction}, respectively. 

For each $b$-flexible component $H$, 
the family $\{ S \in \sgpartb{G}{b}{\alpha}  : S \subseteq V(H) \}$ 
is a partition of $V(H)$, 
and thus 
$\sgpartb{G}{b}{\alpha}$ can be regarded as 
providing the inner canonical structure of each $b$-flexible component. 
If $b = 1$ and $\alpha = -$, 
then this family provides the nontrivial aspect of 
the known general Kotzig-Lov\'asz decomposition for $1$-factorizable graphs. 

As is also the case in bidirected graphs, 
this inner structure $\{ S \in \sgpartb{G}{b}{\alpha}  : S \subseteq V(H) \}$  
is determined in the context of the entire graph $G$ 
and is a refinement of $\sgpartb{H}{b|_H}{\alpha}$.

\begin{ac} 
This work was partly supported by JSPS KAKENHI 15J09683. 
\end{ac}

\bibliographystyle{splncs03.bst}
\bibliography{%
signkl.bib}

\end{document}